\documentclass[preprint,12pt]{elsarticle}

\usepackage{amssymb}
\usepackage{latexsym,amssymb}
\usepackage{enumerate, verbatim}
\usepackage{graphics}
\usepackage[latin1]{inputenc}
\usepackage[T1]{fontenc}
\usepackage{amsmath, amsfonts, amsthm, amscd}

\newtheorem{theorem}{Theorem}[section]

\newtheorem{prop}[theorem]{Proposition}

\newtheorem{corollary}[theorem]{Corollary}
\newtheorem{remark}[theorem]{Remark}
\newtheorem{defi}[theorem]{Definition}

\newcommand{\CC} {\ensuremath{\mathcal{C}}}
\newcommand{\DD} {\ensuremath{\mathcal{D}}}

\newcommand{\C} {\ensuremath{\mathbb{C}}}
\newcommand{\N} {\ensuremath{\mathbb{N}}}
\newcommand{\Z} {\ensuremath{\mathbb{Z}}}
\newcommand{\F} {\ensuremath{\mathbb{F}}}
\newcommand{\wt} {\textnormal{\textrm{wt}}}
\newcommand{\dd} {\textnormal{\textrm{d}}}
\newcommand{\ddd} {\textnormal{\textrm{dd}}}
\newcommand{\Aut} {\textnormal{\textrm{Aut}}}

\newcommand{\SSS} {\textnormal{\textrm{S}}}

\journal{Finite Fields and their Applications}

\begin{document}

\begin{frontmatter}

\title{On involutions in extremal self-dual codes\\ and the dual distance of semi self-dual codes.}

\author{Martino~Borello}

\address{Member
INdAM-GNSAGA (Italy)\\
Dipartimento di Matematica e
Applicazioni\\ Universit\`{a} degli Studi di Milano Bicocca\\
20125 Milan, Italy\\ e-mail: martino.borello@unimib.it}

\author{Gabriele~Nebe}

\address{Lehrstuhl D f\"ur Mathematik,\\ RWTH Aachen University,\\ 52056 Aachen, Germany,\\ e-mail: nebe@math.rwth-aachen.de.
}

\begin{abstract}
A classical result of Conway and Pless is that a natural projection
of the fixed code of an automorphism of odd prime order of a
self-dual binary linear code is self-dual \cite{ConwayPless}. In
this paper we prove that the same holds for involutions under some
(quite strong)
conditions on the codes.\\
In order to prove it, we introduce a new family of binary codes: the
semi self-dual codes. A binary self-orthogonal code is called semi
self-dual if it contains the all-ones vector and is of codimension 2
in its dual code. We prove upper bounds on the dual distance of semi
self-dual codes.\\
 As an application we get the
following: let $\CC$ be an extremal self-dual binary linear code of
length $24m$ and $\sigma \in \Aut(\CC )$ be a fixed point free
automorphism of order 2. If $m$ is odd or if $m=2k$ with
$\binom{5k-1}{k-1}$ odd then $\CC$ is a free $\F_2\langle \sigma
\rangle $-module. This result has quite strong consequences on the
structure of the automorphism group of such codes.
\end{abstract}

\begin{keyword}
semi self-dual codes \sep bounds on minimum distance \sep
automorphism group \sep free modules \sep extremal codes
\end{keyword}

\end{frontmatter}

\section{Introduction}

The research in this paper is motivated by the study of involutions
of extremal self-dual codes, which plays a fundamental role in
\cite{Nebefree,Baut6,BorelloWillems,BDN,BE8,YY}.

Let $m\in \N$ and  $\CC = \CC ^{\perp} \leq \F_2^{24m}$ be an
extremal binary self-dual code, so $d(\CC) = 4m+4$ \cite{MS}. Then
$\CC $ is doubly even \cite{Rshad}. There are unique extremal
self-dual codes of length 24 and  48 and these are the only known
extremal codes of length $24m$. It is an intensively studied open
question raised in \cite{Sloane}, whether an extremal code of length
72 exists. A series of many papers has shown that if such a code
exists, then its automorphism group $\Aut(\CC ) = \{ \sigma \in
S_{24m} \mid \sigma (\CC ) = \CC \}$ has order $\leq 5$ (see
\cite{Thesis} for an exposition of this result).
 Stefka Bouyuklieva
\cite{Bouyuklieva} studies automorphisms of order 2 of such codes.
She shows that if $\CC $ is an extremal code of length $24m$, $m\geq
2$ and $\sigma \in \Aut(\CC ) $ has order 2, then the permutation
$\sigma $ has no fixed points, with one exception, $m=5$, where
there might be 24 fixed points. If $\sigma = (1,2)\ldots ,
(24m-1,24m)$ is a fixed point free automorphism of a doubly even
self dual code $\CC $, then its \textit{fixed code}
$$\CC (\sigma )
:= \{ c\in \CC \mid \sigma (c) = c \} $$ is isomorphic to
$$\pi (\CC (\sigma ) )  = \{ (c_1,\ldots , c_{12m}) \in \F _2^{12m}  \mid
(c_1,c_1,c_2,c_2,\ldots , c_{12m},c_{12m} ) \in \CC \} $$ such that
$$\pi (\{ c+\sigma (c) \mid c\in \CC \} ) = \pi (\CC (\sigma )) ^{\perp }
\subseteq \pi (\CC (\sigma )) .$$ As $\CC $ is doubly-even, all
words in  $\pi (\CC (\sigma ))$ have even weight. It is shown in
\cite{Nebefree} and \cite{BorelloWillems} that the code $\CC $ is a
free $\F_2\langle \sigma \rangle $-module, if and only if $\pi (\CC
(\sigma ) ) $ is self-dual. If $\pi (\CC (\sigma )) $ is not
self-dual then it contains the dual $\DD ^{\perp }$ of some code
$\mathcal{D}$ of length $12m$ with
$${\bf 1} :=
(1,\ldots , 1) \in \pi(\CC(\sigma))^{\perp} \subseteq \mathcal{D}
\subseteq \mathcal{D}^{\perp } \subseteq \pi(\CC(\sigma)).$$
In particular $\dd (\DD ^{\perp }) \geq \dd (\pi (\CC (\sigma )))
= \frac{1}{2} \dd ( \CC (\sigma )) \geq \frac{1}{2} \dd (\CC )$.

\begin{defi}
A binary self-orthogonal code $\DD\subseteq \DD^{\perp}\leq \F_2^n$
of length $n$ is called \textit{semi self-dual}, if ${\bf 1} :=
(1,\ldots , 1)  \in \DD $ and
 $\dim (\DD^{\perp}/\DD ) = 2$.
\end{defi}

Self-orthogonal codes always consist of words of even weight, so
$\wt(c):=|\{i \ | \ c_i=1\}| \in 2\Z $ for all $c\in \DD $. Hence
already the condition that ${\bf 1} \in \DD $ implies that the
length $n=12m$ of $\DD $ is even. Note that $\DD ^{\perp } \subseteq
{\bf 1}^{\perp}  = \{ c\in \F_2^n \mid \wt (c) \in 2\Z \}$ implies
that also $\DD ^{\perp} $ consists of even weight vectors.
The \textit{dual distance} of $\DD $ is the minimum weight of the
dual code $\ddd(\DD) := \dd (\DD ^{\perp }) :=\min(\wt(\DD^{\perp}
\setminus \{0\}))$.

In this paper we will bound the dual distance $\ddd(\mathcal{D}) =
\dd(\mathcal{D}^{\perp })$ of semi self-dual codes. In particular if
the length of $\mathcal{D}$ is $12 m$ with either $m$ odd or $m=2\mu
$ such that $\binom{5\mu -1}{\mu -1}$ is odd, then
$\ddd(\mathcal{D}) \leq 2m $ (see Theorem \ref{main} below for the
general statement).

Then we may conclude the following Theorem.

\begin{theorem}\label{extodd}
Let $\CC =\CC ^{\perp } \leq \F_2^{24m}$ be an extremal code of
length $24m$ and $\sigma \in \Aut(\CC )$ be a fixed point free
automorphism of order 2. Then $\CC $ is a free $\F_2\langle \sigma
\rangle $-module if $m$ is odd or if $m=2\mu$ with
$\binom{5\mu-1}{\mu-1}$ odd.
\end{theorem}

In particular, for $m=3$, we obtain \cite[Theorem 3.1]{Nebefree}
without appealing to the classification of all extremal codes of
length 36 in \cite{Gaborit} and without any serious
computer calculation.

\begin{remark} In {\rm \cite{Zhang}}, Zhang proved that extremal self-dual binary
linear codes of length a multiple of $24$ may exist only up to
length $3672=153\cdot 24$. About $72\%$ of these lengths are covered
by Theorem \ref{extodd}. In particular the projections of fixed
codes by fixed point free involutions in self-dual $[96,48,20]$ and
$[120,60,24]$ codes {\rm(}see {\rm \cite{BWY,BDW}} for an exposition
of the state of the art for the codes with these parameters{\rm)}
are self-dual.
\end{remark}

The same arguments as in \cite{Nebefree} can now be applied to
obtain the following quite strong consequence on the structure of
the automorphism group of such extremal codes.

\begin{corollary}
Let $m \geq 3$ be odd and assume that $m\neq 5$.
 Let \mbox{$\CC=\CC^{\perp} \leq \F_2^{24m}$} be an extremal code.
If $8$ divides $| \Aut(\CC) | $ then a Sylow $2$-subgroup of
$\Aut(\CC)$ is isomorphic to $C_2\times C_2\times C_2$, $C_2\times
C_4$ or $D_8$.
\end{corollary}

\begin{proof}
Let $S$ be a Sylow-2-subgroup of $\Aut(\CC)$.\\
By our assumption and \cite{Bouyuklieva}
 all elements of order $2$ in $\Aut(\CC)$ act
without fixed points on the places $\{ 1,\ldots , 24m\}$. This
immediately implies that all $S$-orbits have length $|S|$, so
$|S|$ divides $24m$ and hence $|S| = 8$.\\
So we only need to exclude $S=C_8$ and $S=Q_8$. This is done by
considering the module structure of $\CC$ as an $\F_2S$-module. Note
that both groups have a unique elementary abelian subgroup, say $Z$,
and $Z\cong C_2$. By Theorem \ref{extodd} the module $\CC$ is a free
$\F_2Z$-module. Chouinard's Theorem \cite{Chouinard} states that a
module is projective if and only if its restriction to every
elementary abelian subgroup is projective. Then $\CC$ is also a free
$\F_2S$-module of rank
$${\rm rk}_{\F_2S} (\CC) = \frac{\dim_{\F_2}(\CC)}{|S|} = \frac{12m}{8} =  3 \cdot \frac{m}{2} \not\in \N $$
a contradiction.
\end{proof}

\begin{remark} Note that the cyclic group $C_8$ is already excluded by the
Sloane-Thompson Theorem {\rm(}see also {\rm \cite{Annika})} because
$S\cong C_8 $ acting fixed point freely on $24m$ points implies that
$S$ is not in the alternating group, so $S$ does not fix any
doubly-even self-dual code.
\end{remark}

\section{Bounds on the dual distance of semi self-dual codes}

In the previous section we introduced the definition of semi
self-dual codes. Now we will prove upper bounds on their dual
distance. Even if this family of codes was introduced as a tool for
the proof of Theorem \ref{extodd}, it seems to be interesting also
by itself.
Applying the methods from \cite{Rshad}, we show the following
theorem.

\begin{theorem}\label{main}
Let $\DD \leq \F_2^n$ be a semi self-dual code. Then the dual
distance of $\DD $ is bounded by
$$\ddd(\DD ) = \dd (\DD ^{\perp} ) \leq \left\{ \begin{array}{ll}
4 \lfloor \frac{n}{24} \rfloor +2 & \mbox{ if } n\equiv 0,2,4,6,8,10,12,14 \pmod{24} \\
4 \lfloor \frac{n}{24} \rfloor +4 & \mbox{ if } n\equiv 16,18,20 \pmod{24}  \\
4 \lfloor \frac{n}{24} \rfloor +6 & \mbox{ if } n\equiv 22
\pmod{24}.
\end{array} \right. $$
If $n=24\mu$ for some integer $\mu$ and
 $\DD$ is doubly-even or
$\binom{5\mu-1}{\mu-1}$ is odd then
$$\ddd(\DD )=\dd({\mathcal D}^{\perp}) \leq 4\mu .$$
\end{theorem}

Theorem \ref{main} follows by combining Remark \ref{selfdual},
Proposition \ref{de}, Proposition \ref{nde} and Proposition
\ref{last}.

\begin{remark}
The well-known Kummer's theorem on binomial coefficients implies
that $\binom{5\mu-1}{\mu-1}$ is odd if and only if there are no
carries when $4\mu$ is added to $\mu-1$ in base $2$.
\end{remark}

By direct calculations with {\sc Magma}, using a database
\cite{database} of all self-dual binary linear codes of length up to
$40$, most of the bounds of Theorem \ref{main} can be shown to be
sharp. In particular, we have semi self-dual codes such that their
dual codes have parameters \ $[4,3,2]$, \ $[6,4,2]$, \ $[8,5,2]$, \
$[10,6,2]$, \ $[12,7,2]$, \ $[14,8,2]$, \ $[16,9,4]$, $[18,10,4]$,
$[20,11,4]$ and $[22,12,6]$ and a doubly-even semi self-dual code
with dual code of parameters $[24,13,4]$.\\

\section{Self-dual subcodes}

From now on let $\DD$ be a semi self-dual code of even length $n\geq
4$.  Furthermore, let $\mu=\left\lfloor\frac{n}{24}\right\rfloor$.

\begin{remark} \label{selfdual}
There are exactly three self-dual codes  $\CC _i = \CC_i^{\perp }$
{\rm(}$i\in\{1,2,3\}${\rm)} with
$$\DD \subset \CC_1,\CC_2,\CC_3 \subset \DD^{\perp}.$$
From the bound on $\dd(\CC _i)$ given in  \textnormal{\cite[Theorem
5]{Rshad}} we obtain
$$
\ddd(\DD ) = \dd (\DD ^{\perp} ) \leq \dd(\CC _1) \leq \left\{
\begin{array}{ll}
4 \mu +6 & \mbox{ if } n\equiv 22 \pmod{24}  \\
4 \mu +4 & \mbox{ otherwise.}
\end{array} \right. $$
\end{remark}

\noindent We aim to find a better bound.

\section{Shadows: the doubly-even case}

\begin{prop}\label{de}
If $\DD$ is doubly-even, then
$$\dd(\mathcal{D}^\perp)\leq \left\{\begin{array}{ll} 4 \mu & \text{if} \ n \equiv 0 \pmod{24} \\ 4\mu+2  & \text{if} \ n\equiv 4,8,12 \pmod{24} \\
4\mu+4 & \text{if} \ n\equiv 16,20 \pmod{24}. \end{array} \right.$$
\end{prop}

\begin{proof}
Since every doubly-even binary linear code is self-orthogonal,
$\mathcal{D}^\perp$ cannot be doubly-even and so in
$\mathcal{D}^\perp$ there exists a codeword of weight $w\equiv 2
\pmod4$. Thus we can take
$\mathcal{D}<\mathcal{F}=\mathcal{F}^\perp<\mathcal{D}^\perp$ with
$\mathcal{F}$ not doubly-even, so that $\mathcal{D}=\mathcal{F}_0 :=
\{ f\in {\mathcal F} \mid \wt (f)\equiv 0 \ \pmod4 \} $ is the
maximal doubly-even subcode of $\mathcal{F}$.

Let $\SSS(\mathcal{F}):=\mathcal{D}^\perp - \mathcal{F}$ denote the
shadow of ${\mathcal F}$. By \cite{BG2004},
\begin{equation}\label{eq1} 2\dd(\mathcal{F})+\dd(\SSS(\mathcal{F}))\leq 4+\frac n 2.\end{equation}
Note that $\dd(\mathcal{D}^\perp)=\min
\{\dd(\mathcal{F}),\dd(\SSS(\mathcal{F})) \}$, since
$\mathcal{D}^\perp = \SSS(\mathcal{F}) \cup \mathcal{F}$. Since we
have the bound \eqref{eq1}, the maximum for $\min
\{\dd(\mathcal{F}),\dd(\SSS(\mathcal{F}))\}$ is reached if
$$\dd(\mathcal{D}^\perp)=\dd(\mathcal{F})=\dd(\SSS(\mathcal{F}))=\left\lfloor
\frac{4+\frac n 2}{3}\right\rfloor$$ so that
$$\dd(\mathcal{D}^\perp) \leq \left\lfloor\frac{8+n}{6}\right\rfloor,$$
which yields the proposition since $\dd(\mathcal{D}^\perp)$ is even.
\end{proof}

In \cite{RainsB} Rains proved more general bounds on the dual
distance of doubly-even binary linear codes,
 without assuming that they contain the all-ones vector.

$$\begin{tabular}{|c|c|c|}
  \hline
  Length & Rains' bound & Our bound \\
  \hline
$24\mu$   & $4\mu+4$ &   $\textbf{4}\mu$\\
$24\mu+4 $& $4\mu+2$ &   $4\mu+2$\\
$24\mu+8 $& $4\mu+4$ &   $\textbf{4}\mu\textbf{+2}$\\
$24\mu+12$& $4\mu+2$ &   $4\mu+2$\\
$24\mu+16$& $4\mu+4$ &   $4\mu+4$ \\
$24\mu+20$& $4\mu+4$ &   $4\mu+4$ \\
  \hline
\end{tabular}$$

With our additional assumption there is a substantial improvement in
particular for lengths divisible by 24.

\section{Weight enumerators: the non doubly-even case.}

In this section  we assume that $\DD $ is not doubly-even. We will
use the following notation:
\begin{itemize}
  \item $N:=\frac n 2$, $2d:=d(\DD ^{\perp })$;
  \item $A(x,y):= W_{\DD }(x,y) = \sum _{c\in \DD} x^{n-\wt(c)} y^{\wt (c)}=x^{2N} + \sum_{i=d}^{N-d}a_i x^{2N-2i}y^{2i} + y^{2N}$ the
weight enumerator of $\mathcal{D}$;
  \item $D(x,y) :=A(\frac{x+y}{\sqrt{2}},\frac{x-y}{\sqrt{2}})=\frac{1}{2} x^{2N} + \sum_{i=d}^{N-d}d_i
x^{2N-2i}y^{2i} + \frac{1}{2} y^{2N}$, so that $2D$ is the weight
enumerator of $\mathcal{D}^{\perp}$;
  \item $B(x,y):=A(x,y)-D(x,y)=\frac{1}{2} x^{2N} + \sum_{i=d}^{N-d}b_i x^{2N-2i}y^{2i} + \frac{1}{2} y^{2N} $;
  \item $F(x,y) := B\left(\frac{x+y}{\sqrt{2}} , i \frac{x-y}{\sqrt{2}}\right)=\frac{1}{2}\left(W_{\SSS(\mathcal{D})}(x,y)-W_{\SSS(\mathcal{D})}\left(\frac{1+i}{\sqrt
2} x,\frac{1-i}{\sqrt 2} y\right)\right)$, where
$\SSS(\mathcal{D})=\mathcal{D}_0^\perp-\mathcal{D}^\perp$ is the
shadow of $\mathcal{D}$.
\end{itemize}

The polynomial $B(x,y)$ is anti-invariant under the MacWilliams
transformation $H: (x,y) \mapsto 1/\sqrt{2}(x+y,x-y)$ and invariant
under the transformation $I:(x,y)\mapsto (x ,-y)$,
 so by \cite[Lemma 3.2]{Bachoc}
$$B(x,y) \in
(x^4-6x^2y^2+y^4) \cdot \C [ x^2+y^2, x^2y^2(x^2-y^2)^2].$$ and we
can write
\begin{equation}\label{eqB}
B(x,y)=(x^4-6x^2y^2+y^4)\cdot \sum_{i=0}^{\lfloor \frac
{N-2}{4}\rfloor} e_i
(x^2+y^2)^{N-2-4i}(x^2y^2(x^2-y^2)^2)^i\end{equation} and,
consequently,
\begin{equation}\label{eqF}
F(x,y)=2(x^4+y^4)\cdot\sum_{i=0}^{\lfloor \frac {N-2}{4}\rfloor} e_i
(2xy)^{N-2-4i}\left(-\frac 1 4 x^8+\frac 1 2 x^4y^4-\frac 1 4
y^8\right)^i.\end{equation}

Notice that \eqref{eqF} implies that the degrees of the monomials of
$F(x,y)$ are congruent to $N-2 \pmod4$. Since
$$\begin{array}{rl}F(x,y)&=\frac{1}{2}\left(W_{\textnormal{S}(\mathcal{D})}(x,y)-W_{\textnormal{S}(\mathcal{D})}\left(\frac{1+i}{\sqrt
2} x,\frac{1-i}{\sqrt 2} y\right)\right)=\\
& =\frac{1}{2}\left(W_{\textnormal{S}(\mathcal{D})}(x,y)-i^N
W_{\textnormal{S}(\mathcal{D})}\left(x,-iy\right)\right),\end{array}$$
it is easy to see that $F(x,y)$ is the weight enumerator of the
following set
$$\mathcal{S}:=\{s\in \SSS(\mathcal{D}) \ | \ \wt(s)\equiv N-2  \pmod4\}.$$
So the coefficients of $F(x,y)$ are non-negative integers.

 Then we
get the following.

\begin{corollary}\label{corcor}
Let $e_i$ be as in \eqref{eqB} and \eqref{eqF} and put $\epsilon
_i:=(-1)^i 2^{N-1-6i} e_i$. Then all $\epsilon _i$ are non-negative
integers.
\end{corollary}

\begin{proof}
We have
$$F(1,y) =(1+ y^4)y^{N-2}\cdot \sum _{i=0}^{\lfloor
\frac{N-2}{4}\rfloor} \epsilon _i y^{-4i} (1-y^4)^{2i}.$$ with
$\epsilon _i:=(-1)^i 2^{N-1-6i} e_i$. Substitute $\lfloor
\frac{N-2}{4}\rfloor-i=h$.
$$F(1,y) =y^{N-2-4\lfloor \frac{N-2}{4}\rfloor}(1+ y^4)(1-y^4)^{2\lfloor \frac{N-2}{4}\rfloor}\cdot \sum _{h=0}^{\lfloor
\frac{N-2}{4}\rfloor} \epsilon_{\lfloor \frac{N-2}{4}\rfloor -h}
(y^{4} (1-y^4)^{-2})^h.$$ Let $r:=N-2-4\lfloor
\frac{N-2}{4}\rfloor$. Note that $r$ is the remainder of the
division of $N-2$ by $4$.
\begin{align*}F(1,y)&=\sum_{j=0}^{2N} f_j y^j = f_0+\ldots+f_{r-1}y^{r-1}+y^r\sum_{j=r}^{2N}f_jy^{j-r}\\
&=y^r(1+ y^4)(1-y^4)^{2\lfloor \frac{N-2}{4}\rfloor}\cdot \sum
_{h=0}^{\lfloor \frac{N-2}{4}\rfloor} \epsilon_{\lfloor
\frac{N-2}{4}\rfloor -h} (y^{4} (1-y^4)^{-2})^h.\end{align*} Then
$f_j=0$ if $j\not\equiv r \pmod 4$. Set $Z=y^4$. Then
$$\sum_{k} f_{4k+r} Z^{k}=(1+ Z)(1-Z)^{2\lfloor \frac{N-2}{4}\rfloor}\cdot \sum
_{h=0}^{\lfloor \frac{N-2}{4}\rfloor} \epsilon_{\lfloor
\frac{N-2}{4}\rfloor -h} (Z (1-Z)^{-2})^h.$$ Put
$$f(Z):=(1+ Z)^{-1}(1-Z)^{-2\lfloor \frac{N-2}{4}\rfloor}, \quad
g(Z):=Z(1-Z)^{-2}.$$ Then there are coefficients $\gamma_{h,k}$ such
that
$$Z^kf(Z)=\sum_{h=0}^{\lfloor \frac{N-2}{4}\rfloor}\gamma_{h,k}g(Z)^h.$$

Since $g(0)=0$ and $g'(0)\neq 0$, we can apply the
B\"{u}rmann-Lagrange theorem (see \cite[Lemma 8]{Rshad}) to obtain
$$\gamma_{h,k}=[\text{coeff. of} \ Z^{h-k} \ \text{in} \ (1-Z)^{-1-2\left\lfloor \frac{N-2}{4} \right\rfloor+2h}] = {{2\lfloor\frac{N-2}{4} \rfloor -h-k}\choose{h-k}} >0.$$
In particular
$$\epsilon_{\left\lfloor \frac{N-2}{4} \right\rfloor-h}=\sum_{k=0}^{\lfloor \frac{h-r}{4} \rfloor} \gamma_{h,k} f_{4k+r}$$
is a non-negative integer for all $h$.
\end{proof}

\begin{prop}\label{nde}
If $\DD $ is not doubly-even and $n \equiv 0,2,4,6,8,10,12,14
\pmod{24} $ then $\dd(\mathcal{D}^\perp)\leq 4\mu+2$.
\end{prop}

\begin{proof}
We have that
\begin{align*} B(1,Y) & = 1/2 + \sum _{j=d}^{N-d}
b_j Y^j + 1/2 Y^{N}\\ & = (1-6Y + Y^2)(1+Y)^{N-2}\cdot \sum
_{i=0}^{\lfloor \frac {N-2}{4} \rfloor} e_i (Y(1-Y)^2(1+Y)^{-4})^i.
\end{align*} Let
$$f(Y):=(1-6Y + Y^2)^{-1}(1+Y)^{2-N}, \qquad
g(Y):=Y(1-Y)^2(1+Y)^{-4}.$$ As before we find coefficients
$\alpha_{i}(N)$ such that
$$f(Y)=\sum_{i=0}^{\lfloor \frac{N-2}{4}
\rfloor}\alpha_{i}(N)g(Y)^i.$$ Then, for $i<d$,
$$e_i= \frac{1}{2}
\alpha_{i}(N).$$

Since $g(0)=0$ and $g'(0)\neq 0$, we can apply the
B\"{u}rmann-Lagrange theorem, in the version of \cite[Lemma
8]{Rshad}, to compute
$$\alpha _i(N) = \text{coeff. of} \ Y^{i} \ \text{in} \
\frac{Yg'(Y)}{g(Y)} f(Y) \left( \frac{Y}{g(Y)} \right) ^i  =: \star
$$ We compute
$$ \star =  (1+Y)^{1-N+4i}(1-Y)^{-2i-1} = (1-Y^2)^{-2i-1} (1+Y)^{2+6i-N }.$$
As $(1-Y^2)^{-2i-1} $ is a power series in $Y^2$ with positive
coefficients, we see that $\alpha _i(N)$ is positive if $2+6i-N >
0$, so if $i > \frac{N-2}{6}$. For $i<d$ we know that $\alpha _i(N)
= 2 e_i = (-1)^i 2^{-N+2+6i} \epsilon _i $ where $\epsilon _i$ is a
non-negative integer, so $\alpha _i(N)$ is not positive for odd
$i<d$.

Write $N=12\mu+\rho$ with $0\leq \rho \leq 7$ and assume that
$d>2\mu+1$. Then $\alpha _{2\mu+1} >0 $ because
$6(2\mu+1)+2-(12\mu+\rho) = 8-\rho
> 0 $ which is a contradiction. We conclude that $d \leq 2\mu+1$ for
$\rho=0,1,2,3,5,6,7$.
\end{proof}

We aim to find an analogous result to Proposition \ref{de} for semi
self-dual codes of length $24\mu$. So we need to find the bound
$\ddd ({\mathcal D}) \leq 4\mu $ also for not doubly even semi-self
dual codes $\DD $ of length $24\mu$. For certain values of $\mu$, we
may show that some coefficient of $F(x,y)$ is not integral.

\begin{prop}\label{last}
If $\DD $ is not doubly-even and $n=24\mu$ with
$\binom{5\mu-1}{\mu-1}$ odd then $\dd(\mathcal{D}^\perp)\leq 4\mu$.
\end{prop}

\begin{proof}
With the notations used above, we get
\begin{align*}
\alpha _{2\mu}(12\mu) & = \text{coeff. of} \ Y^{2\mu} \ \text{in} \
 (1-Y^2)^{-4\mu-1} (1+2Y+Y^2)   \\
&=\text{coeff. of} \ Z^{\mu} \ \text{in} \  (1-Z)^{-4\mu-1}  +
\text{coeff. of} \ Z^{\mu-1} \ \text{in} \  (1-Z)^{-4\mu-1}   \\ &=
{{5\mu}\choose{\mu}} + {{5\mu-1}\choose{\mu-1}} = 6
{{5\mu-1}\choose{\mu-1}}.
\end{align*}
On the other hand, assuming that $\dd(\DD^{\perp}) \geq 4\mu+2$, we
have
$$
\alpha _{2\mu}(12\mu) = 2 e_{2\mu} = 2^{2} \epsilon _{2\mu}.$$ As
$\epsilon _{2\mu}$ is a non-negative integer, we get that
${{5\mu-1}\choose{\mu-1}} $ is even.
\end{proof}

It seems to be impossible to obtain the same bound for the other
values of $\mu$ by just looking at weight enumerators. For $\mu=5$
(the first value for which ${{5\mu-1}\choose{\mu-1}} $ is even), we
get examples of $\{e_i\}$ for which $F(x,y)$ has non-negative
integer coefficients and $B(1,y)=1/2+\textnormal{O}(y^{22})$. From
one of these we computed
$W_{\mathcal{D}}(1,y)=1+\textnormal{O}(y^{22})$,
$W_{\mathcal{D}^\perp}(1,y)=1+\textnormal{O}(y^{22})$ and
$W_{\SSS(\mathcal{D})}(1,y)=\textnormal{O}(y^{18})$, all with
non-negative integer coefficients.

\section*{Acknowledgements}

Both authors are indebted to the Dipartimento di Matematica e
Applicazioni, Universit\`{a} degli Studi di Milano-Bicocca, and the
Lehrstuhl D f\"ur Mathematik, RWTH Aachen University, for
hospitality and excellent working conditions, while this paper has
mainly been written.

This paper is partially part of the PhD thesis \cite{Thesis} of the
first author who expresses his deep gratitude to his supervisors
Francesca Dalla Volta and Massimiliano Sala.

\end{document}